\documentclass[reqno]{amsart}
\usepackage{amsmath,amsthm,amssymb}
\usepackage{latexsym}
\usepackage{eucal}

\newtheorem{theorem}{Theorem}[section]
\newtheorem{lemma}{Lemma}[section]
\newtheorem{corollary}{Corollary}[section]

\theoremstyle{definition}

\def\cal{\mathcal}

\begin{document}
\title[An evolutions equation as the WKB correction  \ldots]{An evolutions equation
as the WKB correction in long-time asymptotics of Schr\"odinger
dynamics\footnote{\tiny\rm  This is a preprint of an article whose
final and definitive form has
  been published in Comm. Partial Differential Equations, copyright of Taylor
  and Francis, available online at http://www.informaworld.com
}}
\author{Sergey A. Denisov}

\address{University of Wisconsin-Madison,
Mathematics Department, 480 Lincoln Dr., Madison WI, 53706-1388,
USA {\rm e-mail: denissov@math.wisc.edu}} \maketitle

\begin{abstract}
We consider the $3$--dimensional Schr\"odinger operator with
slowly decaying potential whose {\bf radial} derivatives are
short-range. The long-time asymptotics for solution of the
corresponding non-stationary equation is established. In this
case, the standard WKB-correction should be replaced by the
solution of certain evolution equation.
\end{abstract} \vspace{1cm}

 Consider the Schr\"odinger operator
\begin{equation}
H=-\Delta+V,\quad x\in \mathbb{R}^3 \label{hamiltonian}
\end{equation}

In this paper, we assume that potential $V(x)$ satisfies the
following conditions:\\

{\bf Conditions A:}
\[
 |V|< Cr^{-\gamma},
\left|\frac{\partial V}{\partial r}\right|<Cr^{-1-\gamma},
\left|\frac{\partial^2 V}{\partial
r^2}\right|<Cr^{-1-2\gamma},V(x)\in C^2(\mathbb{R}^3), \quad r=|x|
\]
and
\begin{equation}
1> \gamma>1/2
\end{equation}
We will study the scattering properties of $H$, in particular the
long-time asymptotics of the group $e^{itH}$. This group gives the
solution to the non-stationary equation: if $\psi(t)=e^{itH}f$ and
$f\in L^2(\mathbb{R}^3)$, then
\[
 \partial \psi/\partial t=iH\psi,\quad \psi(0)=f
\]
The scattering theory for Schr\"odinger operator is a classical
subject (see, e.g. \cite{Yafaev1, Yafaev2, Derezinski}). We want
to mention here several results that motivated us to write this
note. Some of them are rather old and well-known, some are quite
new.

The following theorem can be found in \cite{Yafaev1} (Theorem
4.1), \cite{Yafaev3} (see also \cite{dollard, buslaev, DG} for
other results on long-range scattering).
\begin{theorem}{\rm (Yafaev \cite{Yafaev1,
Yafaev3})}\label{long-range-theorem} Assume that
\[
|D^k V(x)|<C\langle x\rangle^{-\gamma-|k|}
\]
with $\gamma>1/2$ and any $k: |k|<k_0$, $k_0$--big enough.
Introduce
\begin{equation}
\Xi(x,t)=|x|^2/(4t)-t\int_0^1 V(sx)ds \label{long-range}
\end{equation}
and
\begin{equation}
\cal{E}_0(t)[f](x)=\exp\left[i\Xi(x,t)\right](2it)^{-3/2}\hat{f}(x/(2t))
\end{equation}
where $\hat{f}$ denotes the Fourier transform of $f$.
 Then, the limits
\begin{equation}
\cal{W}^{\pm}=s-\lim_{t\to\pm \infty} \exp(iHt)\cal{E}_0(t)
\end{equation}
exist. They are called the modified wave
operators\,\,\footnote{Actually, any $\gamma\in (0,1)$ was treated
but the modification $\Xi$ is then more involved}.
\end{theorem}
Note that the modification is done in the physical space by the
integral of potential over the segments $\{ tx, t\in [0,1]\}$. We
have
\[
\exp\left[i\Xi(x,t)\right]=\exp\left[i|x|^2/(4t)\right]\cdot
\exp\left[-it\int_0^1 V(sx)ds\right]
\]
and the second factor will be called the  WKB-correction in the
long-time asymptotics of $\exp(itH)$. For the short-range case
(i.e., when $\gamma>1$) this factor can be discarded.
 Perhaps, the definition of wave (modified wave) operators in momentum space is
 more standard
(see, e.g. \cite{Yafaev1, Hormander}). Nevertheless, in our paper,
we will be working in the physical space only.

The existence of modified wave operators per se does not imply
that the spectrum of  $H$ is purely absolutely continuous on
positive half-line. Therefore, the following result (see, e.g.,
\cite{Yafaev1}, Theorem 4.2 and references therein) complements
the Theorem~\ref{long-range-theorem}.
\begin{theorem} \label{mourre}
Assume that $V(x)$ is such that
\[
\left|\frac{\partial V}{\partial r}\right|<Cr^{-1-}
\]
Then, the positive spectrum of $H$ is absolutely continuous.
\end{theorem}
The proof of this result is based on the Mourre estimates and also
shows that the absorption principle holds in the suitable weighted
spaces. It also allows to handle the short-range perturbations.
Other proofs (see, e.g., \cite{BP1, S}) are using various PDE
techniques and can handle more general cases.

In the meantime, we are not aware of any results where the
existence of modified wave operators is proved in the case when
the strong decay is assumed only for the radial derivative of the
potential and nothing is known about the tangential component of
the gradient. If, instead, the tangential component of the
gradient is decaying fast, then the recent result by Perelman
\cite{Perelman} shows that the a.c. spectrum of $H$ fills
$\mathbb{R}^+$ but it can coexist with singular spectrum. In
\cite{Perelman}, the WKB correction for the spatial asymptotics of
Green's function is proved to be similar to that in the
Theorem~\ref{long-range-theorem}: the modification comes from the
integration of potential over the segments. In papers \cite{H1,
H2}, the scattering theory for Hamiltonians with zero-degree
homogeneous potentials was considered. The dynamics in this case
is nontrivial and the analysis is quite intricate.

In this paper, we consider potential that satisfies Conditions A.
The Theorem~\ref{mourre} guarantees that the positive spectrum is
purely absolutely continuous. Our goal is to obtain an analog of
Theorem~\ref{long-range-theorem}. In our opinion, the most
interesting part of this paper is new WKB correction to the
asymptotics\footnote{C. G\'{e}rard kindly mentioned to us that
somewhat similar corrections were obtained in Sections 3.6 and 4.8
of \cite{Derezinski}.}. This correction is not so surprising in
view of recent calculations done for sparse slowly decaying
potentials (see discussion in \cite{Denisov1}).

The paper is divided into two parts. In the first one, we
introduce and study the evolution equation which gives the right
WKB correction. The second part contains the proof of the main
Theorem, several corollaries and examples. We will use the
following notations: $\langle x\rangle$ stands for
$(|x|^2+1)^{1/2}$, $\cal{F}$ denotes the Fourier transform
normalized to be unitary map from $L^2(\mathbb{R}^3)$ onto itself,
the symbol $D$ denotes the differential of a function, $\Sigma$
stands for the two-dimensional unit sphere, $B$ denotes the
(positive) Laplace-Beltrami operator on $\Sigma$. The symbols
$H^\alpha(\Sigma)$ stand for Sobolev spaces with index $\alpha$.
For any vector $x\neq 0, x\in \mathbb{R}^3$, we write
$\hat{x}=x/|x|$. For smooth function $f(x), x\in \mathbb{R}^3$,
$f_r$ will mean the radial derivative and $r=|x|$. Symbol
$\cal{D}(A)$ stands for the domain of definition of self-adjoint
operator $A$ and $E_\lambda(A)$ is its
orthoprojector.\vspace{0.5cm}

{\bf Acknowledgement.} We are grateful to A. Ionescu and A.
Kiselev for useful discussions and to  C. G\'{e}rard, I. Herbst,
and D. Yafaev for providing us with some very important
references.

\section{Evolution equation}
In this section, we discuss one very special evolution equation
which will play an important role later. It can be easily studied
by well-known methods but we will give all details for
completeness.

Consider the following evolution equation:

\begin{equation}\label{evolution}
iky_\tau(\tau,\theta)=\frac{(By)(\tau,\theta)}{\tau^2}+V(\tau,\theta)y(\tau,\theta),
\tau>0
\end{equation}
where $k\in \mathbb{R}\backslash \{0\}$, $V(\tau,\theta)$ is
real-valued, and the function $y(\tau,\theta)\in L^2(\Sigma)$ for
any $\tau> 0$. Later on, we will let
$V(\tau,\theta)=V(\tau\cdot\theta)$ where the potential $V(x)$ is
the Schr\"odinger potential introduced before. Let us define and
study $U(k,\tau_0,\tau)f$-- the solution of (\ref{evolution})
satisfying an initial condition $U(k,\tau_0, \tau_0)f=f$ where
$\tau, \tau_0> 0$ and $f\in L^2(\Sigma)$. Note first that if
$V=0$, then we can write the solution in terms of more familiar
Schr\"odinger evolution on the sphere:
\[
U_0(k,\tau_0,\tau)=\exp\left[ \frac{1}{ik}
\frac{\tau-\tau_0}{\tau\tau_0} B\right]
\]
For any fixed $f\in L^2(\Sigma)$, the following limit exists
\[
\lim_{\tau\to\infty} U_0(k,\tau_0,\tau)f=\exp\left[ \frac{1}{ik}
\frac{B}{\tau_0} \right]f
\]
Then, $U(k,\tau_0,\tau)$ allows the formal Duhamel expansion
involving $U_0(k,s,t)$ and $V(s)$. Assuming
$\|V(\tau,\cdot)\|_{L^\infty(\Sigma)}\in L^1_{\rm
loc}(\mathbb{R}^+)$ one can easily prove convergence of this
expansion in $L^2(\Sigma)$ and we {\bf define} $U(k,\tau_0,\tau)$
in this way. For simplicity, from now on we assume that $V\in
C^2((0,\infty)\times \Sigma)$. Then, for $f\in H^2(\Sigma)$,
$U(k,\tau_0,\tau)f$ gives the actual solution to differential
equation. Notice that $U(k,\tau_0,\tau)$ is unitary on
$L^2(\Sigma)$ and
$U^*(k,\tau_0,\tau)=U^{-1}(k,\tau_0,\tau)=U(k,\tau,\tau_0)$. This
is an elementary property of any evolution equation with
self-adjoint coefficients. If $V(\tau,\theta)$ decays fast enough,
then the perturbation theory easily yields existence of the strong
limit for $U(k,\tau_0,\tau)$ as $\tau\to +\infty$. For instance,
we have the following elementary Lemma

\begin{lemma}\label{l-1}
If $\|V(\tau,\cdot)\|_\infty\in L^1[1,\infty)$, the the limit
\[
\lim_{\tau\to\infty} U(k,\tau_0,\tau)f=U(k,\tau_0,\infty)f
\]
exists for any $f\in L^2(\Sigma)$.
\end{lemma}
\begin{proof}
If $U(k,\tau_0,\tau)=U_0(k,\tau_0,\tau)Y(\tau)$, then
\[
Y_\tau(\tau)=U_0(k,\tau_0,\tau)^{-1}V(\tau)U(k,\tau_0,\tau)Y(\tau)
\]
That yields $\|Y\|$ is bounded in $\tau$. Since $\|Y_\tau\|\in
L^1[1,\infty)$, the Cauchy criteria guarantees that $Y(\tau)$
converges to $Y(\infty)$ in norm topology.  $U_0(k,\tau_0,\tau)$
converges to $U_0(k,\tau_0,\infty)$ strongly and so we get the
statement of the Lemma.
\end{proof}

Consider the following conditions that are essentially identical
to Conditions A:

{\bf Conditions A1:}
\begin{equation}\label{V-estimates}
 \|V(\tau,\cdot)\|_{L^\infty(\Sigma)}< C\tau^{-\gamma},
\left\|\frac{\partial V}{\partial
\tau}(\tau,\cdot)\right\|_{L^\infty(\Sigma)}<C\tau^{-1-\gamma},
\left\|\frac{\partial^2 V}{\partial
\tau^2}(\tau,\cdot)\right\|_{L^\infty(\Sigma)}<C\tau^{-1-2\gamma}
\end{equation}
and
\begin{equation}
V(\tau,\theta)\in C^2([1,\infty)\times\Sigma), \tau\geq 1,1>
\gamma>1/2
\end{equation}
 As we will see later, the solution to evolution equation
does not have to have any limit as $\tau\to\infty$ in this case.
In the meantime, one can obtain some estimates on the decay of
various derivatives.  From now on, the norm $\|\cdot\|$ means
$L^2(\Sigma)$ norm. For any $f\in L^2(\Sigma)$, denote
\begin{equation}
W(k,\tau)f=U(k,1,\tau)f \label{doublev}
\end{equation}

\begin{lemma}\label{main-lemma}
Assume that $V$ satisfies Conditions A1 and  $k\in I\subset
\mathbb{R}^+$, $f\in H^4(\Sigma)$. Then, the following estimates
hold true
\begin{equation}\label{der-tau}
\|W_\tau(k,\tau)f\|<C\tau^{-\gamma}
\end{equation}
\begin{equation}\label{der-tau-tau}
\|W_{\tau \tau}(k,\tau)f\|<C \tau^{-2\gamma}
\end{equation}
\begin{equation}\label{der-k}
\|W_k (k,\tau)f\|<C\tau^{1-\gamma}
\end{equation}
\begin{equation}\label{der-tau-k}
\|W_{\tau k}(k,\tau)f\|<C \tau^{1-2\gamma}
\end{equation}
\begin{equation}\label{der-k-k}
\|W_{kk}(k,\tau)f\|<C \tau^{2-2\gamma}
\end{equation}
where the constant $C$ depends on $I$, $\|f\|_{H^4(\Sigma)}$, and
constants in the Conditions A1.
\end{lemma}
\begin{proof}
Let $y=W(k,\tau)f$ and $u=y_\tau$. We have
\begin{equation} \label{eq1}
iky_\tau=\left[\frac{B}{\tau^2}+V\right]y,\,ik
u_\tau=\left[\frac{B}{\tau^2}+V\right]u+\left[V_\tau
-2\frac{B}{\tau^3}\right]y
\end{equation}
Multiply the first
equation by $2\tau^{-1}$ and add to the second one to get
\[
ik\left[ u_\tau+2\frac{u}{\tau}\right]=\left[\frac{B}{\tau^2}+V
\right]u+\left[V_\tau +2\frac{V}{\tau}\right]y
\]
Let
\begin{equation}
\phi=\tau^2 u \label{denote-u}
\end{equation}
 Then the equation can be rewritten as
\begin{equation}\label{phi-eq}
ik\phi_\tau=\left[\frac{B}{\tau^2}+V
\right]\phi+\tau^2\left[V_\tau +2\frac{V}{\tau}\right]y,\,
\phi(1)=(Bf+V(1)f)/(ik)
\end{equation}
Therefore,
\[
\phi(\tau)=U(k,1,\tau)\phi(1)+\int_1^\tau s^2 U(s,\tau)\left[
V_\tau+2\frac{V}{s}\right]yds
\]
Using Conditions A1 and $\|y\|=\|f\|$, we get an estimate
\[
\|\phi(\tau)\|<C\|f\|_{H^2(\Sigma)}+C\|f\|\tau^{2-\gamma}
\]
Thus, (\ref{denote-u}) yields (\ref{der-tau}). To prove the
estimate on the second derivative in time, we introduce
$\psi=y_{\tau \tau}$. Then,

\[
ik\psi_\tau=\left[
\frac{B}{\tau^2}+V\right]\psi+\left[V_{\tau\tau}y+2V_\tau
y_\tau-4\frac{B}{\tau^3}y_\tau+6\frac{B}{\tau^4}y\right]
\]
Multiply the second equation in (\ref{eq1}) by $4\tau^{-1}$ and
add to the equation above. We get
\[
ik\left[\psi_\tau+4\frac{\psi}{\tau}\right]=\left[\frac{B}{\tau^2}+V\right]\psi+\left[
V_{\tau \tau}y+2V_\tau y_\tau-2\frac{B}{\tau^4}
y+4\frac{V}{\tau}y_\tau+4\frac{V_\tau}{\tau}y\right]
\]

Using the suitable substitution again, one has
\[
\psi(\tau)=\tau^{-4}U(k,1,\tau)\psi(1)
\]
\[
+\tau^{-4}\int_1^\tau s^4 U(k,s,\tau)\Bigl[ V_{\tau \tau} y\Bigr.
\Bigl. +2V_\tau y_\tau-2\frac{B}{s^4}
y+4\frac{V}{s}y_\tau+4\frac{V_\tau}{s}y\Bigr]ds
\]
and
\[
\psi(1)=(ik)^{-2}(B+V)^2|_{\tau=1}f+(ik)^{-1}(V_\tau(1)-2B)f
\]
Using estimate on $y_\tau$, Conditions A1, and equation
(\ref{evolution}) to bound $s^{-4}By$, we get (\ref{der-tau-tau}).

Differentiate the first equation in (\ref{eq1}) in $k$. If
$y_k=\chi$, then
\begin{equation} \label{eq-2}
ik\chi_\tau=\left[\frac{B}{\tau^2}+V\right] \chi-iy_\tau,
\chi(1)=0
\end{equation}
and
\[
\chi(\tau)=-i\int_1^\tau U(k,s,\tau)y_\tau(s)ds
\]
Now, the estimate (\ref{der-tau}) yields (\ref{der-k})
immediately.

 Let us estimate the mixed derivative of $y$, i.e.
 $\mu=y_{k\tau}$. To do so, differentiate (\ref{eq-2}) in $\tau$:
 \[
 ik\mu_\tau=\left[\frac{B}{\tau^2}+V\right]\mu +\left[
 -iy_{\tau\tau}-2\frac{B}{\tau^3}y_k+V_\tau y_k\right]
 \]
Multiply (\ref{eq-2}) by $2\tau^{-1}$ and add to the last
equation. We then have
\[
ik\left[\mu_\tau+2\frac{\mu}{\tau}\right]=\left[
\frac{B}{\tau^2}+V\right]\mu+\left[-iy_{\tau\tau}+V_\tau
y_k+2\frac{V}{\tau}y_k-2i\frac{y_\tau}{\tau}\right]
\]
which yields
\[
\mu(\tau)=\tau^{-2}\mu(1)+\tau^{-2} \int_1^\tau s^2 U(k,s,\tau)
\left[ -iy_{\tau \tau} +V_\tau
y_k+2\frac{V}{s}y_k-2i\frac{y_\tau}{s}\right]ds
\]
Using estimates on $y_{\tau \tau}, y_k, y_\tau$ and Conditions A1,
we get (\ref{der-tau-k}).

The last estimate on $y_{k k}$ can be obtained by differentiating
(\ref{eq-2}) in $k$. If $\eta=y_{kk}$, then
\[
ik\eta_\tau=\left[\frac{B}{\tau^2}+V\right]\eta-2iy_{\tau k},
\eta(1)=0
\]
Estimate (\ref{der-tau-k}) readily yields (\ref{der-k-k}).
\end{proof}

Making some assumptions on angular derivatives of $V$ we obtain
quite a different result.

\begin{lemma}\label{wkb-scalar}
Assume that $V$ is such that
\begin{equation}\label{add-cond}
\|V\|_{L^\infty(\Sigma)}\in L^\infty [1,\infty), \|D_\theta
V(\tau, \theta)\|_{L^\infty(\Sigma)}<C \tau^{-\gamma},
\|D^2_\theta V(\tau, \theta)\|_{L^\infty(\Sigma)}<C
\tau^{1-2\gamma}
\end{equation}
Then, the following asymptotics holds
\[
[W(k,\tau)f](\theta)=\exp\left[(ik)^{-1}\int_1^\tau
V(s,\theta)ds\right]\cdot [W_{mod}(k,\tau)f](\theta)
\]
where $W_{mod}(k,\tau)f\to W_{mod}(k,\infty)f$ in $L^2(\Sigma)$ as
$\tau\to\infty$ and $f\in L^2(\Sigma)$.
\end{lemma}
\begin{proof}
The proof is elementary. Let $y=W(k,\tau)f$. First, let us get an
estimate on $Dy$. We have
\[
ik[Dy]_\tau=\frac{B}{\tau^2} [Dy]+V[Dy]+y[DV], [Dy](1)=Df
\]
and
\begin{equation}\label{Dy}
\|Dy\|<C\|f\|_{H^1}+C\int_1^\tau s^{-\gamma}ds<C\tau^{1-\gamma}
\end{equation}
Denote
\[
\varphi=(ik)^{-1}\int_1^\tau V(s,\theta)ds
\]
and write $W(k,\tau)f=\exp(\varphi)\nu(\tau)$. For $\nu$, we have
equation
\[
ik \nu_\tau=\frac{B}{\tau^2}\nu+g,\, g=\frac{2D\nu \cdot
D\varphi+B\varphi+D\varphi\cdot D\varphi}{\tau^2}
\]
and $\nu(1)=f$. But $D\nu =D(y\exp(-\varphi))$. From
(\ref{add-cond}) and (\ref{Dy}), we have
\[
\|g\|<C\tau^{-2\gamma},\,\|g\|\in L^1[1,\infty)
\]
Since $U_0(k,1,\tau)\to U_0(k,1,\infty)$ strongly, we immediately
get the statement of the Lemma.
\end{proof}
The meaning of the Lemma is quite simple: if $V(\tau,\theta)$ does
not depend much on angle $\theta$, then the $\tau^{-2}B$ term in
(\ref{evolution}) can be neglected.

The next result in not so surprising in view of recent papers
\cite{Denisov2, Safronov}.
\begin{lemma}\label{oscillation}
Assume that $V(\tau)=\tau^{-2}[BQ](\tau)$ and both
$V(\tau,\theta)$ and $Q(\tau,\theta)$ satisfy the Conditions A1
with $\gamma>2/3$. Then, for any $f\in L^2(\Sigma)$, we have
\[
U(k,1,\tau)f\to U(k,1,\infty)f, \quad \tau\to\infty
\]
\end{lemma}
\begin{proof}
The equation (\ref{evolution}) can be rewritten as follows
\[
iky_\tau=\frac{B}{\tau^2} \left[y+Qy\right]+g,\quad
g=-Q\frac{B}{\tau^2}y-2\frac{DQ\cdot Dy}{\tau^2}
\]
From Lemma \ref{main-lemma}, we get
$\|\tau^{-2}By\|<C\tau^{-\gamma}$. By interpolation with
$\|y\|=\|f\|$, we have $\|\tau^{-1} Dy\|<C\tau^{-\gamma/2}$. For
$Q$, we have $\|Q\|<C\tau^{-\gamma},
\|\tau^{-2}BQ\|<C\tau^{-\gamma}$. Interpolating again,
$\|\tau^{-1}DQ\|<C\tau^{-\gamma}$. Consequently,
$\|g\|<C\tau^{-3\gamma/2}$ and $\|g\|\in L^1[1,\infty)$. Let us
consider $z=y+Qy$. For $z$,
\[
ikz_\tau=\frac{B}{\tau^2}z +g_1, g_1=g+ik(Qy_\tau+Q_\tau y),
\|g_1\|\in L^1[1,\infty)
\]
Therefore, $z(\tau)\to z(\infty)$. Since
$\|Q(\tau,\theta)\|_\infty\to 0$, the limit of $y$ exists as well.
\end{proof}
We believe that the condition $\gamma>2/3$ can be replaced by
$\gamma>1/2$.

The following example is very instructive.

 {\bfseries Example.} Introduce the spherical coordinates on $\Sigma$ using
 angles $\phi\in [0,2\pi), \psi\in [0,\pi)$. Consider, say, diadic decomposition of $[1,\infty)$.
  For $\tau\in [2^n,
 2^{n+1}]$, we let
 $M(\tau,\phi,\psi)=v(\tau)\sin(2^n\phi)\chi(\psi)$,
 where $v(\tau)$ is such that $|v(\tau)|<C\tau^{-\gamma},
 |v'(\tau)|< C\tau^{-1-\gamma}, |v''(\tau)|<C\tau^{-1-2\gamma}$. We also assume $v(\tau)=0$ for $\tau\in [2^n-1,
 2^n+1]$ and any $n$. Function $\chi(\psi)=0$ for
 $\psi\in [0,\delta]\cup [\pi-\delta, \pi]$ and is infinitely smooth.
For $V=M$ and $\gamma>1/2$, conditions of Lemma \ref{main-lemma}
are satisfied.

Letting $Q=M$ and $V(\tau,\phi,\psi)=\tau^{-2}BQ(\tau,\phi,\psi)$,
we satisfy conditions of the Lemma \ref{oscillation} as long as
$\gamma>2/3$. Therefore, the scalar modification $\exp(\varphi)$
used in Lemma~\ref{wkb-scalar} is not correct if applied to
situation considered in Lemma \ref{oscillation} (i.e. the
statement of Lemma~\ref{wkb-scalar} is wrong under conditions of
Lemma \ref{oscillation}). That can be easily proved by
contradiction if $v(\tau)$ is chosen properly.

It is important to emphasize that Lemmas \ref{l-1} and
\ref{main-lemma} never used the special properties of
Laplace-Beltrami operator. The other two Lemmas did use the fact
that Laplace-Beltrami operator is the second-order differential
operator.

\section{Asymptotics of the Schr\"odinger evolution}

The main goal of this section is to prove an analog of Theorem
\ref{long-range-theorem} provided that only the radial derivatives
are short-range. We need to introduce some notations first. Recall
the definition of $W(k,\tau)$, formula (\ref{doublev}). For any
$t> 0$, consider the following operator
\[
[\cal{E}(t)f](x)=(2it)^{-3/2}\exp\left[i|x|^2/(4t)\right]\cdot
W(|x|/t, |x|)\left[\hat f(|x|/(2t)\theta)\right] (\hat{x})
\]
acting on $f\in \Omega$. The space $\Omega$ consists of functions
$f$ such that $\hat{f}\in C^\infty(\mathbb{R}^3)$, $\hat{f}$ is
compactly supported and vanishes in a neighborhood of the origin
(i.e. $\hat{f}\in C_0^\infty (\mathbb{R}^3\backslash \{0\})$. For
clarity, we explain the meaning of the third factor
\[W(|x|/t, |x|)\left[\hat f(|x|/(2t)\theta)\right] (\hat{x})\]
in more details.
 To define this function on the sphere of radius $|x|$, we
consider $\hat f(|x|/(2t)\theta)$ as the function on the unit
sphere (i.e., $\theta\in \Sigma$). Since $\hat{f}$ vanishes near
the origin, we do not have any problems with definition  at
$|x|=0$. Then, we act on this function with the operator
$W(k,\tau)$ where $k=|x|/t$ and $\tau=|x|$. The resulting function
is taken at the point $\hat{x}\in \Sigma$. This third factor is
the WKB correction in the time evolution of $\exp(itH)$ as
$t\to\infty$. We will need to know some properties of
$\cal{E}(t)$.

\begin{lemma}
For any $t>0$, $\cal{E}(t)$ can be extended to a unitary linear
operator on $L^2(\mathbb{R}^3)$.
\end{lemma}
\begin{proof}
Fix $t>0$. Linearity of $\cal{E}(t)$ on $\Omega$ is obvious. Let
us show that it is an isometry on $\Omega$. We have
\[
\|\cal{E}(t)\|_2^2=\frac{1}{(2t)^3}\int_0^\infty
\rho^2\int_{\theta\in\Sigma} \left|W(\rho/t,\rho)\left[\hat
f(\rho/(2t)\sigma)\right](\theta)\right|^2d\theta d\rho
\]
Since $W(k,\tau)$ is unitary for any $k$ and $\tau$,
\[
\|\cal{E}(t)\|_2^2=\frac{1}{(2t)^3}\int_0^\infty
\rho^2\int_{\sigma\in\Sigma} \left|\hat
f(\rho/(2t)\sigma)\right|^2d\sigma
d\rho=\frac{1}{(2t)^3}\int_{\mathbb{R}^3} |\hat{f}(x/(2t))|^2dx
\]
\[
=\int_{\mathbb{R}^3} |\hat{f}(x)|^2dx=\|f\|^2
\]
Since $\Omega$ is dense in $L^2(\mathbb{R}^3)$, $\cal{E}(t)$ can
be extended to $L^2(\mathbb{R}^3)$ as an isometry. The adjoint
operator can be easily computed
\[
[\cal{E}^*(t)f](x)=e^{i\frac{3\pi}{4}}(2t)^{3/2}\cdot
\cal{F}^{-1}\left\{\exp\left[-it|w|^2\right]\cdot
W^*(2|w|,2|w|t)\left[f(2t|w|\sigma)\right](\hat w)\right\}
\]
The kernel of $\cal{E}^*(t)$ is trivial. Therefore, the range of
$\cal{E}(t)$ is $L^2(\mathbb{R}^3)$ and $\cal{E}(t)$ is unitary.
\end{proof}
For $t<0$, $\cal{E}(t)$ can be defined in the same way. Now we are
prepared for the main result of this paper.
\begin{theorem}\label{main-theorem}
Assume that $V$ satisfies Conditions A. Then, for any $f\in
L^2(\mathbb{R}^3)$, the following limits exist
\[
\cal{W}_{\pm}f=\lim_{t\to \pm \infty} \exp(iHt)\cal{E}(t)f
\]
We will call these operators $\cal{W}_{\pm}$ the modified wave
operators.
\end{theorem}
\begin{proof}
We will consider the case $t\to +\infty$ only. The other situation
is analogous. Since both operators $\exp(iHt)$ and $\cal{E}(t)$
are unitary, the limit $\cal{W}^{+}$, if exists, is an isometry.
Also, it is sufficient to prove existence of $\cal{W}^{+}f$ for
$f\in \Omega$ because $\overline{\Omega}=L^2(\mathbb{R}^3)$. Fix
$f\in \Omega$. We will use Cook's method to show existence of the
limit. To do that, compute a derivative
\[
\frac{d}{d t}\left[ \exp(itH)\cal{E}(t)f\right]=\exp(itH)\left[
iH\cal{E}(t)f+\frac{d}{d t}\cal{E}(t)f\right]
\]
To prove existence of the limit, it is sufficient to show that
this derivative has $L^2(\Sigma)$--norm in $L^1[1,\infty)$. This
is equivalent to
\[
\left\|iH\cal{E}(t)f+\frac{d}{d t}\cal{E}(t)f\right\|\in
L^1[1,\infty)
\]
Checking the last claim is a straightforward calculation that uses
properties of evolution $W(k,\tau)$ studied in the previous
section. For simplicity, write $\left[\cal{E}(t)
f\right](x)=\kappa(t,x) \cdot Y(t,x)$, where
\[
\kappa(t,x)=(2it)^{-3/2}\exp\left[i|x|^2/(4t)\right],\quad
Y(t,x)=W(|x|/t, |x|)\left[\hat f(|x|/(2t)\theta)\right] (\hat{x})
\]
Since
\[
\left(i\frac{\partial}{\partial t}+\Delta\right)\kappa(x,t)=0, t>0
\]
we are left with estimating
\[
\Delta (\kappa Y)+i\frac{d}{d t}(\kappa Y)-V\kappa Y=2\nabla
\kappa \cdot \nabla Y+\kappa \Delta Y+i\kappa \frac{d}{dt}
Y-V\kappa Y
\]
Let us compute each derivative and show the cancelation of the
main terms. We assumed that $f\in \Omega$ and then
$\hat{f}(\omega)$ has support within the spherical layer
$0<\delta_1<|\omega|<\delta_2$. Therefore,  $Y(x,t)=0$ for
$|x|<2\delta_1 t$ and for $|x|>2\delta_2 t$.
\begin{equation}\label{ddr}
2\nabla \kappa \cdot \nabla Y=\kappa
\frac{i|x|}{t}\frac{d}{dr}Y=\kappa \frac{i|x|}{t}\left\{
t^{-1}W_k(|x|/t, |x|)\left[\hat f(|x|/(2t)\theta)\right] (\hat{x})
\right.
\end{equation}
\[
\left. +W_\tau(|x|/t, |x|)\left[\hat f(|x|/(2t)\theta)\right]
(\hat{x})+(2t)^{-1}W(|x|/t, |x|)\left[\hat
f_r(|x|/(2t)\theta)\right] (\hat{x}) \right\},
\]
\[
i\kappa \frac{d}{dt}Y=-i\kappa \frac{|x|}{t^2}W_k(|x|/t,
|x|)\left[\hat f(|x|/(2t)\theta)\right] (\hat{x})
\]
\[
-i\kappa \frac{|x|}{2t^2}W(|x|/t, |x|)\left[\hat
f_r(|x|/(2t)\theta)\right] (\hat{x}),
\]
We can write $\Delta$ in the following way
\[
\Delta=\frac{1}{r^2}\partial_r(r^2
\partial_r)-\frac{B}{r^2}=\partial^2_r+\frac{2}{r}\partial_r-\frac{B}{r^2}
\]
The first derivative $d/dr$ was calculated in (\ref{ddr}). The
second derivative gives
\[
\frac{d^2Y}{dr^2}=t^{-2}W_{kk}(|x|/t, |x|)\left[\hat
f(|x|/(2t)\theta)\right] (\hat{x})+2t^{-1}W_{k \tau}(|x|/t,
|x|)\left[\hat f(|x|/(2t)\theta)\right] (\hat{x})
\]
\[
+W_{\tau\tau}(|x|/t, |x|)\left[\hat f(|x|/(2t)\theta)\right]
(\hat{x})+t^{-2}W_k(|x|/t, |x|)\left[\hat
f_r(|x|/(2t)\theta)\right] (\hat{x})
\]
\[
+t^{-1}W_{\tau}(|x|/t, |x|)\left[\hat f_r(|x|/(2t)\theta)\right]
(\hat{x})+(2t)^{-2}W(|x|/t, |x|)\left[\hat
f_{rr}(|x|/(2t)\theta)\right] (\hat{x})
\]
Then, we can do some obvious cancelations and combine these terms
as follows
\begin{eqnarray}
2\nabla \kappa \cdot \nabla Y+\kappa \Delta Y+i\kappa \frac{d}{dt}
Y-V\kappa Y=
\end{eqnarray}
\[
=\kappa\left\{i\frac{|x|}{t} W_{\tau}(|x|/t,
|x|)-\frac{B}{|x|^2}-V\right\}\left[\hat
f_r(|x|/(2t)\theta)\right] (\hat{x})+R(x,t)
\]
By the definition of evolution equation itself, the first term is
identically zero. Now, we just need to show that
$\|R(\cdot,t)\|\in L^1[1,\infty)$. But $R(x,t)=\kappa
(d^2/dr^2+2r^{-1}d/dr)Y(x,t)$ and we have to use formulas above
and estimates on the derivatives from Lemma \ref{main-lemma}. It
is more convenient to write these bounds in spherical coordinates:
\[
\|R\|^2<Ct^{-3}\int\limits_{2\delta_1 t}^{2\delta_2 t}\rho^2
\|\left[d^2/d\rho^2+2\rho^{-1}d/d\rho\right]Y(\rho,\theta,t)\|^2_{L^2(\theta\in
\Sigma)}d\rho<Ct^{-4\gamma}
\]
where the last estimate follows directly from Lemma
\ref{main-lemma} because $\hat{f}(\omega)$ is infinitely smooth.
The constant $C$ does depend on $\delta_{1(2)}$. Since
$\gamma>1/2$, we have $\|R\|\in L^1[1,\infty)$ and that finishes
the proof.
\end{proof}

 Notice that this
Theorem is strictly stronger than Theorem
\ref{long-range-theorem}.  This is due to Lemmas \ref{wkb-scalar},
\ref{oscillation}, and the Example considered at the end of
previous section. It says, essentially, that the right WKB
correction is provided by solution to the evolution equation
(\ref{evolution}). We do not see how any further asymptotical
expansion of $W(k,\tau)$ can be done (i.e. $W(k,\tau)$ seems to be
``thing in itself", e.g. like the function $e^{ix}$). That,
apparently, is the key difficulty in this and many related
problems.

An important example of $V$ from Theorem \ref{main-theorem} is
provided by the following construction which is similar to Example
considered in the previous section.  Again, take a diadic
decomposition of $[1,\infty)$ and define $V$ in spherical
coordinates as follows: $V(\rho,\theta)=v(\rho)\cdot h_n(\theta)$
for $\rho\in [2^n, 2^{n+1}]$, where $v(\rho)$ is borrowed from
Example and $h_n(\theta)$-- smooth function on $\Sigma$, different
for each $n$.

Now, let us study operators $\cal{W}_{\pm}$.

\begin{lemma}For any bounded measurable function $\Phi(\lambda)$, we have an
intertwining property
\begin{equation}\label{intertwining}
\Phi(H)\cal{W}_{\pm}=\cal{W}_{\pm} \Phi(H_0)
\end{equation}
The subspaces $\cal{H}_{\pm}={\rm Ran}\, \cal{W}_{\pm}$ reduce $H$
and $H|_{\cal{H}_{\pm}}$ are unitarily equivalent to $H_0$.
\end{lemma}
\begin{proof}
Again, we give details for $t\to +\infty$ only. First, consider
$\Phi(\lambda)=\exp(iT\lambda)$ with fixed $T$. Let $f\in \Omega$.
Then, (\ref{intertwining}) is equivalent to
\[
\lim_{t\to+\infty} \exp(itH) \cal{E}(t-T) f=\lim_{t\to+\infty}
\exp(itH) \cal{E}(t)\exp(iTH_0)f
\]
We have
\[
[\cal{E}(t-T)f](x)=(2i(t-T))^{-3/2}\exp\left[i|x|^2/(4t-4T)\right]\times
\]
\[
 W(|x|/(t-T), |x|)\left[\hat f(|x|/(2t-2T)\theta)\right]
(\hat{x})
\]
and
\[
[\cal{E}(t)\exp(iTH_0)f](x)=(2it)^{-3/2}\exp\left[i|x|^2/(4t)+iT|x|^2/(4t^2)\right]\times
\]
\[
W(|x|/t, |x|)\left[\hat f(|x|/(2t)\theta)\right] (\hat{x})
\]
Using the estimate $|x|<a t, a>0$, identity
\[
\Bigl[W(|x|/(t-T), |x|)-W(|x|/t, |x|)\Bigr]\hat
f(|x|/(2t)\theta)=\int\limits_{|x|/t}^{|x|/(t-T)} W_k(\alpha,
|x|)\hat f(|x|/(2t)\theta)d\alpha
\]
and the bound (\ref{der-k}), we easily get
\[
\|\cal{E}(t-T) f-\cal{E}(t)\exp(iTH_0)f\|_{L^2(\mathbb{R}^3)}\to 0
\]
as $t\to \infty$. If (\ref{intertwining}) is true for dense set
$\Omega$, then it must be true for $L^2(\mathbb{R}^3)$ since the
operators involved are bounded. Thus, (\ref{intertwining}) holds
for $\Phi(\lambda)=\exp(iT\lambda)$ with any $T$. The standard
approximation argument then implies (\ref{intertwining}) in the
general setting. We refer the reader to Theorem 4, p. 69,
\cite{Yafaev2} for details. In particular,
$E_\lambda(H)\cal{W}_{\pm}=\cal{W}_{\pm} E_\lambda (H_0)$ and
$\cal{W}_{\pm }\cal{D}(H)\subseteq \cal{D}(H_0)$. Of course, both
$\cal{D}(H)$ and $\cal{D}(H_0)$ are identical to
$H^2(\mathbb{R}^3)$.

Taking adjoint of (\ref{intertwining}), we get
\begin{equation}
E_\lambda(H_0)\cal{W}_{\pm}^*=\cal{W}_{\pm}^* E_\lambda(H)
\end{equation}
and therefore $\rm {Ker} \,\cal{W}_{\pm}^*$ reduces
$E_\lambda(H)$. That implies $\rm {Ran} \,\cal{W}_{\pm}$ reduces
$E_\lambda(H)$ as well. The rest of the Lemma now follows easily.
\end{proof}
The following corollary is straightforward
\begin{corollary}
For any $f=\cal{W}_{\pm}g\in \cal{H}$, the following asymptotics
holds true
\begin{equation}
\exp(-itH) f=\cal{E}(t)g+\bar{o}(1)
\end{equation}
where $\|\bar{o}(1)\|_{L^2(\mathbb{R}^3)}\to 0$ as
$t\to\pm\infty$.
\end{corollary}


\begin{thebibliography}{99}

\bibitem{buslaev} V.S. Buslaev, V.B. Matveev,  Wave operators for the
Schr\"odinger equation with slowly decreasing potential, (Russian.
English summary), Teoret. Mat. Fiz., 2, 1970, no. 3, 367--376.

\bibitem{H1} H. Cornean, I. Herbst, C. G\'{e}rard,   Spiraling
attractors and quantum dynamics for a class of long-range magnetic
fields, MaPhySto Report 2002-36, to appear in J. Funct. Anal.

\bibitem{Denisov1} S.A. Denisov, Wave propagation through sparse potential barriers,
(preprint).

\bibitem{Denisov2} S.A. Denisov, Absolutely continuous spectrum for multidimensional
Schr\"odinger operators, IMRN, 2004, no. 74, 3963--3982.

\bibitem{Derezinski} J. Derezi\'{n}ski, C. G\'{e}rard,
Scattering theory of classical and quantum $N$-particle systems,
Texts and Monographs in Physics, Springer-Verlag, Berlin, 1997.

\bibitem{DG} J. Derezi\'{n}ski, C. G\'{e}rard, Long-range scattering
in the position representation,  J. Math. Phys.,  38,  1997, no.
8, 3925--3942.

\bibitem{dollard} J. Dollard,
Asymptotic convergence and the Coulomb interaction, J.
Mathematical Phys.,  5,  1964, 729--738.

\bibitem{Hormander} L. H\"ormander, The existence of wave
operators  in scattering theory, Math. Z., 146, 1976, 69--91.

\bibitem{H2} I. Herbst, E. Skibsted,  Quantum scattering for potentials
independent of $\vert x\vert $: asymptotic completeness for high
and low energies,  Comm. Partial Differential Equations,  29,
2004,  no. 3-4, 547--610.

\bibitem{S} T. Ikebe, Y. Saito,  Limiting absorption method and
absolute continuity for the Schr\"odinger operator,  J. Math.
Kyoto Univ.,  12,  1972, 513--542.

\bibitem{Perelman} G. Perelman,  Stability of the absolutely continuous
spectrum for multidimensional Schr\"odinger operators,  Int. Math.
Res. Not.,  2005,  no. 37, 2289--2313.

\bibitem{BP1} B. Perthame, L. Vega,  Morrey-Campanato
estimates for Helmholtz equations,  J. Funct. Anal., 164, 1999,
no. 2, 340--355.

\bibitem{Safronov} O. Safronov, On the a.c. spectrum of
multi-dimensional Schr\"odinger operators with slowly decaying
potentials, Comm. Math. Phys., Vol. 254, 2005, no. 2, 361--366.

\bibitem{Yafaev1} D. Yafaev,
Scattering theory: some old and new problems,
 Lecture Notes in Mathematics, 1735, Springer-Verlag, Berlin,
 2000.

\bibitem{Yafaev2} D. Yafaev,  Mathematical scattering theory, General
theory,
 Translations
of Mathematical Monographs, 105, American Mathematical Society,
Providence, RI, 1992

\bibitem{Yafaev3} D. Yafaev, Wave operators for Schr\"odinger
equation, Theoret. and Math. Phys., 45, no. 2, 1981, 992--998.



\end{thebibliography}
\end{document}